\documentclass[10pt]{amsart}

\newcommand{\SL}{\textup{SL}}
\newcommand{\GL}{\textup{GL}}
\newcommand{\ch}{\textup{ch}}

\DeclareMathOperator{\Sort}{sort}

\renewcommand{\ge}{\geqslant}
\renewcommand{\le}{\leqslant}

\usepackage{Style}
\usepackage{doi}
\usepackage{ytableau}
\usepackage{stackrel}
\usepackage[table]{xcolor}
\usepackage{subcaption}
\usepackage{stmaryrd}
\newcommand{\Qbinom}[2]{\genfrac{[}{]}{0pt}{}{#1}{#2}}
\newcommand{\qbinom}[2]{\genfrac{\llbracket}{\rrbracket}{0pt}{}{#1}{#2}}
\DeclareMathOperator{\LR}{LR}
\DeclareMathOperator{\SSYT}{SSYT}
\newcolumntype{C}{{>{\centering\arraybackslash$}p{2.7em}<{$}}}
\DeclareMathOperator{\Out}{Out}

\title[Schur log-concavity and the quantum Pascal triangle]{Schur log-concavity and\\ the quantum Pascal triangle}
\author{\'Alvaro Guti\'errez and Christian Krattenthaler}
\address{University of Bristol, Bristol, UK}
\email{a.gutierrezcaceres@bristol.ac.uk}
\address{Universit\"at Wien, Vienna, Austria}
\email{Christian.Krattenthaler@univie.ac.at}
\date{\today}

\begin{document}

\begin{abstract}
We say a sequence $f_0, f_1, f_2, \ldots$ of symmetric functions is \emph{Schur log-concave} if $f_n^2 - f_{n-1}f_{n+1}$ is Schur positive for all $n\ge1$. We conjecture that a very general class of sequences of Schur functions
satisfies this property, and show it for sequences of Schur functions indexed by partitions with growing first part and column. Our findings are related to work of
Lam, Postnikov and Pylyavskyy on Schur positivity,
and of Butler, Sagan, and the second author on $q$-log-concavity.
\end{abstract}

\maketitle

\section{Introduction}

We say a sequence $f_0, f_1, f_2, \ldots$ of symmetric functions is \emph{Schur log-concave} if $f_n^2 - f_{n-1}f_{n+1}$ is Schur positive for all $n\ge1$. It is \emph{strongly Schur log-concave} if $f_nf_{n+i} - f_{n-1}f_{n+i+1}$ is Schur positive for all $i\ge0$ and $n\ge1$.
We put forward the following conjecture. 
(We refer the reader to \S\ref{sec:preliminaries} for all terminology
and notation.)
\begin{conjecture}\label{main conjecture}
    Let $\lambda, \beta$ be partitions, $\alpha$ be an integral vector.
    Suppose $\ell(\lambda) \gg \ell(\alpha)$ and $\lambda_{\ell(\lambda)} \gg \beta_1$. Then, the sequence
    \[
    s_{\lambda},~~
    s_{\lambda\cup\beta+\alpha},~~
    s_{\lambda\cup^2\beta+2\alpha}, ~~
    s_{\lambda\cup^3\beta+3\alpha}, ~~\ldots
    \]
    is strongly Schur log-concave.
\end{conjecture}

We adopt the convention that $s_{\gamma} = 0$
if $\gamma$ is not a partition. Hence, the sequence in the conjecture might be finite or infinite. 
For $|\lambda|\le6$ and $\ell(\lambda)\le3$, the conjecture holds for all tested $\alpha$ and $\beta$ such that $\ell(\lambda) > \ell(\alpha)$ and $\lambda_{\ell(\lambda)} \ge \beta_1$. If $\ell(\lambda) = \ell(\alpha)$, however, the claim is no longer true: consider e.g.~$\lambda = (3,3),~ \alpha = (1,1),~ \beta=(3)$.
\medskip

Our goal is to provide evidence supporting the conjecture. 
When $\lambda = \beta = \varnothing$ and $\ell(\alpha)=1$, 
the conjecture is true by the Jacobi--Trudi formula
(see Corollary~\ref{cor:h} below).
When $\lambda$ is a hook, $\beta=\varnothing$, and $\ell(\alpha) = 1$, 
the conjecture is essentially known ---it follows from either \cite[Thm.~3.1(1)]{BergeronBiagioliRosas}, \cite[Prop.~3.3]{BergeronMcNamara}, or \cite[Thm.~5]{Lam}.

More generally, we show Conjecture~\ref{main conjecture} for the case 
$\beta_1 \le 1$, $\ell(\alpha) \le 1$ of growing first part and first column.
\begin{mainthm}\label{main:first row}
The following sequence is strongly Schur log-concave:
\[s_{\lambda},~~
    s_{\lambda\cup(1^j)+(k)},~~
    s_{\lambda\cup(1^{2j})+(2k)},~~
    s_{\lambda\cup(1^{3j})+(3k)},~~ \ldots\] 
for any $k = 0$ or $k \ge \lambda_2$, and any $j = 0$ or $j\ge  \lambda'_2$.
\end{mainthm}

To our knowledge, Schur log-concavity of sequences has not been studied explicitly. However, the Schur positivity of expressions of the form $s_\mu s_\nu - s_\rho s_\delta$ has an extensive literature: see for instance \cite{Orellana, Lam} and the references therein. Although sequences are not usually discussed, some of these results can be reframed in our language. 
The aforementioned case when $\lambda$ is a hook,
    $\beta=\varnothing$ and $\ell(\alpha)=1$
 appears in this context as a particular case of a conjecture of \cite{FFLP}, which was ultimately established in \cite{Lam}. 

Symmetric polynomials are characters of $\GL_n$ representations.
A related problem is to study sequences of $\SL_n$ characters, which are the images of symmetric polynomials in $\ZZ[x_1, \ldots x_n]/(x_1\cdots x_n - 1)$. For $\SL_2$, the variable $q$ is usually preferred; $\ZZ[q,q^{-1}] \cong \ZZ[x_1,x_2]/(x_1x_2-1)$. 
A \emph{quantum integer} is defined as
\(
[n] = \frac{q^{-n}-q^{n}}{q^{-1}-q} = q^{-n+1} + q^{-n+3} + \cdots + q^{n-3} + q^{n-1},
\)
for $n\ge1$, with $[0] = 0$, and a \emph{quantum binomial} as
\[
\Qbinom{n}{k} = 
\frac{[n][n-1]\cdots[n-k+1]}{[k][k-1]\cdots[1]}
\]
for $n\ge k \ge0$. They are Laurent polynomials in $\N_0[q,q^{-1}]$. Moreover, they are centred ---invariant under $q\mapsto q^{-1}$--- and unimodal
---that is, the coefficients are non-negative, (weakly) monotone increasing up to a certain point, and from there on (weakly) monotone decreasing. They can be realised via the following specialisation of Schur polynomials:
\begin{equation}\label{eq:plethysm Q-char}
\Qbinom{n}{k} = s_{(1^k)}(q^{-n+1}, \ldots, q^{n-3}, q^{n-1}).
\end{equation}

A sequence $a_0(q), a_1(q), a_2(q), \ldots$ of (Laurent) polynomials in $\N_0[q,q^{-1}]$ is \emph{Schur log-concave} if for $n\ge1$ the difference
$a_{n}(q)^2 - a_{n+1}(q)a_{n-1}(q)$ is a unimodal centred polynomial with non-negative coefficients. It is \emph{strongly Schur log-concave} if
$a_{n+i}(q)a_{n}(q) - a_{n+i+1}(q)a_{n-1}(q)$
is a unimodal centred polynomial with non-negative coefficients for $n\ge1$ and $i\ge0$.

Our next conjecture, based on data gathered with SageMath \cite{sagemath}, concerns diagonals
in the quantum Pascal triangle (see Figure \ref{fig:quantumPascal}).
\begin{conjecture}\label{conj: diagonals}
For all choices of $n\ge k\ge 0$, $\alpha\ge-1$, and $\beta \ge 0$, the diagonal of the quantum Pascal triangle
starting at $\Qbinom{n}{k}$ and of slope $\alpha/\beta$, namely
    \[
    \Qbinom{n}{k},
    \Qbinom{n-\alpha}{k + \beta},
    \Qbinom{n-2\alpha}{k + 2\beta},
    \Qbinom{n-3\alpha}{k + 3\beta},
    \Qbinom{n-4\alpha}{k + 4\beta}, \ldots,
    \]
    is strongly Schur log-concave.
\end{conjecture}

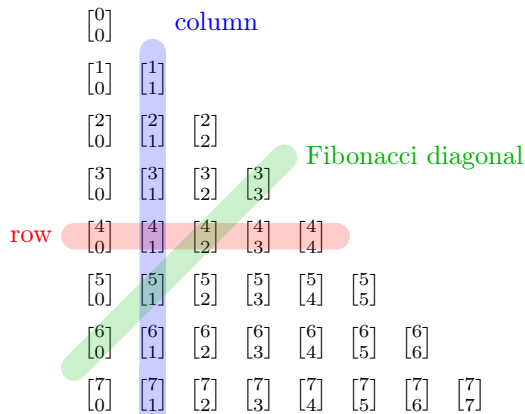
\begin{figure}
    \centering
    \begin{tikzpicture}[x=2em, y=2em]
        \foreach \i in {0,...,7}{
        \foreach \j in {0,...,\i}{
            \node (\i\j) at (\j,-\i) {$\Qbinom{\i}{\j}$};
        }}
        \begin{scope}[line cap = round, line width = 10, opacity=.2]
            \draw[red] (44.east)--(40.west) node[anchor=east, opacity=1] {row};
            \draw[blue] (61.south)--(71.south)--(11.north) node[anchor=south west, opacity=1] {column};
            \draw[green!70!black] (60.south west)--(33.north east) node[anchor=west, opacity=1] {Fibonacci diagonal};
        \end{scope}
    \end{tikzpicture}
    \caption{The first eight rows of the quantum Pascal triangle. A row, a column, and a Fibonacci diagonal are highlighted.}
    \label{fig:quantumPascal}
\end{figure}

The case $(n,k,\alpha,\beta) = (n,0,0,1)$ of the conjecture refers to the rows. 
The case $(n,k,\alpha,\beta) = (0,k,-1,0)$ of the conjecture refers to the columns. Both follow from Theorem~\ref{main:first row} 
but are of independent interest.

\begin{coro}\label{thm:rows}
    For all $n$, the $n$th row $\Qbinom{n}{0}, \Qbinom{n}{1}, \ldots, \Qbinom{n}{n}$ of the quantum Pascal triangle is strongly Schur log-concave.
\end{coro}
\begin{coro}\label{thm:cols}
    For all $k$, the $k$th column $\Qbinom{k+0}{k}, \Qbinom{k+1}{k}, \Qbinom{k+2}{k}, \ldots$ of the quantum Pascal triangle is strongly Schur log-concave.
\end{coro}

Conjecture \ref{conj: diagonals} refers to many other studied sequences, such as the $n$th Fibonacci diagonal $(n,k,\alpha,\beta) = (n,0,1,1)$. See Figure \ref{fig:quantumPascal}.
The general form of Conjecture \ref{conj: diagonals} does
however not follow from Conjecture \ref{main conjecture}.
\medskip

Corollaries \ref{thm:rows} and \ref{thm:cols} are related to a series of papers by Butler \cite{Butler}, Sagan \cite{SaganInductive, SaganLogConcave}, and
the second author \cite{Krattenthaler} in the 90s. However, these
authors studied $q$-analogues (as opposed to quantum analogues). A
\emph{$q$-integer}
(without ``quantum'') is defined as \(\llbracket n\rrbracket = 1 + q + q^2  + \cdots + q^{n-1}\), and
\emph{$q$-binomial coefficients} (again without ``quantum'') are realised via the principal specialisation of Schur polynomials,
\(
q^{\binom{k}{2}} \qbinom{n}{k} = s_{(1^k)}(1, q, \ldots, q^{n-1}).
\)

A sequence $a_0(q), a_1(q), \ldots$ of $\N_0[q]$ is $q$-log-concave if $a_n^2(q) - a_{n-1}(q)a_{n+1}(q) \in \N_0[q]$, and strong $q$-log-concavity is defined in the obvious way.
Butler showed (weak) $q$-log-concavity of the  columns of the $q$-Pascal triangle and
the second author the strong version; for us it follows from the main theorem (see Note \ref{note:q-analogue}).
The strong $q$-log-concavity of rows and columns of the $q$-Pascal triangle were shown combinatorially by
Sagan~\cite{SaganLogConcave}. 
Our methods
cannot establish the $q$-log-concavity of the rows in a straightforward way.

\subsection*{Structure of the paper}
In \S\ref{sec:preliminaries} we give some background
on symmetric functions, symmetric polynomials, and representation theory of $\SL_n$. In \S\ref{sec:observations}, we define log-concavity of sequences in any algebra over an ordered ring. Lemma~\ref{lem:homo}
provides a powerful tool for translating log-concavity from one algebra to another. We conclude the section by giving (non-)examples of properties of Schur log-concavity. We
establish Theorem~\ref{main:first row} in \S\ref{sec:main}, and deduce some corollaries. In \S\ref{sec:further}, we show that Conjecture~\ref{main conjecture} satisfies necessary conditions for Schur log-concavity derived from
work of McNamara \cite{McNamara} and
of
the first author and Rosas~\cite{GutierrezRosas}. In \S\ref{sec:quantum} we show the results concerning quantum binomials.

\section{Background}\label{sec:preliminaries}
\subsection{Symmetric functions} We follow \cite{StanleyEC2}.
Let $\mathbf{x} = \{x_1, \ldots, x_n\}$ be a set of variables.
The algebra of symmetric polynomials
in $\mathbf{x}$ is the invariant space $\Lambda[\mathbf{x}] = \Q[\mathbf{x}]^{S_n}$, where the symmetric group acts by permuting variables.
The \emph{elementary} and \emph{complete homogeneous} polynomials are defined by 
\[
e_k(\mathbf{x}) = \sum_{1\le i_1 < i_2 < \cdots < i_k \le n} x_{i_1}x_{i_2}\cdots x_{i_k},
\quad
h_k(\mathbf{x}) = \sum_{1\le i_1 \le i_2 \le \cdots \le i_k \le n} x_{i_1}x_{i_2}\cdots x_{i_k}.
\]
A \emph{partition} $\lambda = (\lambda_1, \ldots, \lambda_k)$ is a weakly decreasing finite sequence of non-negative integers.
The components $\lambda_i$ are called \emph{parts} of the partition~$\lambda$.
The \emph{length} $\ell(\lambda)$ is the number of non-zero parts. We write $\mathrm{Par}_{\le n}$ for the set of partitions of length at most $n$. We usually omit $0$s and use exponents to denote repeated parts, so that e.g.~$(4,4,2,1,1,1,0)$ is
abbreviated by $(4^2,2,1^3)$.
The set of cells of $\lambda$,
also called \emph{Young diagram of~$\lambda$}, is $Y(\lambda) = \{(i,j) \ : \ 1 \le j \le \lambda_i\}$. 
We draw Young diagrams
according to the English convention; the coordinates of the cells are matrix-like. For instance,
the Young diagram of $(6^2,3,2,1)$ is
shown in Figure~\ref{fig: diagram}.

The \emph{transpose} of $\lambda$ is the partition $\lambda'$ such that $Y(\lambda') = \{(i,j) \ : \ 1 \le i \le \lambda_j\}$.
The addition of partitions is entry-wise. This generalises to addition with an integral vector: given a partition $\lambda = (\lambda_1, \ldots, \lambda_k)$ of length $\ell(\lambda) \le k$ and a vector $v\in\ZZ^k$, we let $\lambda+v = (\lambda_1+v_1, \lambda_2+v_2, \ldots, \lambda_n+v_k) \in \ZZ^k$. For a scalar $n\in\N$, we let $n\lambda = (n\lambda_1,\ldots,n\lambda_k)$. The union $\lambda\cup\mu$ of two partitions $\lambda$ and $\mu$ is the partition formed by sorting the parts of $\lambda$ and $\mu$.
We let $\lambda\cup^n\mu = \lambda\cup\mu\cup\cdots\cup\mu$ with $n$ copies of $\mu$.
\medskip

The antisymmetric polynomial
$a_\mu(\mathbf{x})$ indexed by the partition $\mu$ is\break $\det(x_j^{\mu_i})_{i,j}$. Note that letting $\rho = (n-1, \ldots, 2, 1)$
the determinant $a_\rho(\mathbf{x})$ is a Vandermonde determinant.
The \emph{Schur polynomial} $s_\lambda$ indexed by $\lambda$ is the quotient
\[
s_\lambda(\mathbf{x}) = \frac{a_{\lambda+\rho}(\mathbf{x})}{a_\rho(\mathbf{x})}.
\]
Note that \eqref{eq:plethysm Q-char} is a direct consequence of the definitions.

The Schur polynomial has a combinatorial interpretation in terms of tableaux, which is further generalised to a \emph{skew Schur polynomial} below. 
A \emph{skew partition} $\lambda/\mu$ is a pair of partitions such that $\lambda_i \ge \mu_i$ for all $i$. A partition $\lambda$ is identified with the skew partition $\lambda/\varnothing$.
A \emph{Young tableau} of shape $\lambda/\mu$ in the alphabet $[n]$ is a function $T : Y(\lambda)\setminus Y(\mu) \to [n]$. It is \emph{semistandard} if $T(i,j) \le T(i,j+1)$ and $T(i,j) < T(i+1,j)$ whenever the
expressions are defined.
The set of semistandard Young tableaux of shape $\lambda/\mu$ in the alphabet $[n]$ is denoted $\SSYT_n(\lambda/\mu)$.
The \emph{content} $\mathrm{ct}(T)$ is the multiset of its entries; the \emph{weight} of $T$ is $\mathbf{x}^T = \prod_{i \in \mathrm{ct}(T)} x_{i}$. See Figure \ref{fig: tableau} for a semistandard Young tableau of shape $(5,5,3,2,2)/(4,3,1,1)$ and weight $x_1^5x_2^2x_3$. The following result
can be found in \cite[Thm.~7.15.1]{StanleyEC2}.
\begin{figure}
\ytableausetup{smalltableaux}
    \begin{subfigure}{.4\textwidth}
    \centering
    $\ydiagram{6,6,3,2,1}$
    \caption{}
    \label{fig: diagram}
    \end{subfigure}
    \begin{subfigure}{.4\textwidth}
    \centering
    $\ytableaushort{
    {\none}{\none}{\none}{\none}1,
    {\none}{\none}{\none}12,
    {\none}11,
    {\none}2,
    13}$
    \caption{}
    \label{fig: tableau}
    \end{subfigure}
    \caption{}
\end{figure}
\begin{thm}
Let $|\mathbf{x}| = n$. Then
    $s_\lambda(\mathbf{x}) = \sum_{T\in\SSYT_n(\lambda/\varnothing)} \mathbf{x}^T$.
\end{thm}
One can more generally define $s_{\lambda/\mu}(\mathbf{x}) = \sum_{T\in\SSYT_n(\lambda/\mu)} \mathbf{x}^T$.
We direct the reader to \cite{Gessel-Viennot, FK} for combinatorial proofs of the Jacobi--Trudi identities,
given in the next two theorems.
\begin{thm}[\sc Jacobi--Trudi identity]
    \(
    s_\lambda(\mathbf{x}) = \det(h_{\lambda_i-i+j}(\mathbf{x}))_{i,j}.
    \)
\end{thm}
\begin{thm}[\sc Dual Jacobi--Trudi identity]
    \(
    s_{\lambda'}(\mathbf{x}) = \det(e_{\lambda_i-i+j}(\mathbf{x}))_{i,j}.
    \)
\end{thm}
It follows that $s_{(k)} = h_k$ and $s_{(1^k)} = e_k$. Moreover, we obtain the following lemma by collecting the above remarks.
\begin{lem}\label{plethysm Q-char h}
We have    \(
    \Qbinom{k+\ell}{k} = h_\ell(q^{-k},\ldots,q^{k-2},q^{k}).
    \)
\end{lem}

The \emph{algebra of symmetric functions} is the colimit $\Lambda = \lim_{n\to\infty} \Lambda[x_1, \ldots, x_n] = \Q[x_1, x_2, \ldots]^{S_\infty}$. The elementary symmetric function $e_k$, complete homogeneous symmetric function $h_k$, and skew Schur function $s_{\lambda/\mu}$ are the limits of the respective sequences of polynomials $\big(e_k(x_1,\ldots,x_n)\big)_n, \big(h_k(x_1,\ldots,x_n)\big)_n, \big(s_{\lambda/\mu}(x_1,\ldots,x_n)\big)_n$. We similarly consider $\SSYT(\lambda/\mu) = \bigcup_n \SSYT_n(\lambda/\mu)$.

The \emph{Littlewood--Richardson coefficients} $c_{\mu,\nu}^\lambda$ are the coefficients in the
expansion $s_\mu s_\nu = \sum_\lambda c_{\mu,\nu}^\lambda s_\lambda$.
The \emph{reverse reading word} $\mathrm{rrw}(T)$ of a tableau $T$ is obtained by reading the entries of $T$ 
from right to left in each row, and arranging the (ordered) rows from top to bottom. For instance, the reverse reading word of the tableau in Figure \ref{fig: tableau} is $12111231$.
A (reverse reading) word is \emph{Yamanouchi} if every prefix contains
at least as many $1$s as $2$s, at least as many $2$s as $3$s, etc.
For a proof of the next theorem, see e.g.~\cite[\S{A}1.3]{StanleyEC2}.
\begin{thm}[\sc Littlewood--Richardson rule]
    The coefficient $c_{\mu,\nu}^\lambda$ is the cardinality of the set 
    \[
    \LR_{\mu,\nu}^\lambda = \{T \in \SSYT(\lambda/\mu) \ : \ \mathrm{ct}(T) = \nu \text{~and~} \mathrm{rrw}(T) \text{~is Yamanouchi}\}.
    \]
\end{thm}
These objects are called \emph{Littlewood--Richardson tableaux}.

\subsection{Representation theory}
We follow \cite{FH, Paget-Wildon}.
Let $\SL_2 = \SL_2(\CC)$ be the group of $2\times2$ matrices with determinant 1. A \emph{representation} of $\SL_2$ is a vector space $V$ together with a homomorphism $\rho:\SL_2 \to \GL(V)$.
The \emph{Weyl character} $\ch(V)$ of a representation $V$ is 
the unique polynomial in $\N_0[x_1, x_2]^{S_2}/
(x_1x_2-1)\cong\Lambda[q,q^{-1}]$ such that
\[
\mathrm{tr}\Big(\rho\Big(\begin{smallmatrix}
    q & 0 \\
    0 & q^{-1}
\end{smallmatrix}\Big)\Big)
= \ch(V)(q,q^{-1})
\]
for all $q$.
Weyl characters form the Grothendieck group of the category of finite dimensional representations of $\SL_2$.
\begin{ejemplo}
    Let $L(k)$ be the irreducible representation of $\SL_2$ of dimension $k$. Its Weyl character is $q^{-k+1} + q^{-k+3} + \cdots + q^{k-1}$.
\end{ejemplo}

\begin{thm}[\cite{StanleySl2}]\label{thm:Stanley unimod}
    A Laurent polynomial $a(q)$ is unimodal, centred, and has non-negative coefficients if and only if it is the Weyl character of an $\SL_2$ representation.
\end{thm}

The \emph{Schur functor} $S^\lambda$ is a functor from the category of $\GL_n$ representations to itself. It is uniquely characterised by the property
\[
\ch(S^\lambda V) = s_\lambda \circ \ch(V),
\]
where $\circ$ denotes \emph{plethysm}. 
For a definition of plethysm in generality, see \cite[p.~447]{StanleyEC2}. For us, it suffices to provide a rule to compute some cases of it.
\begin{lem}\label{l:rule plethysm}
    Let $g(x_1, \ldots, x_n)$ and $f(y_1, \ldots, y_m)$ be symmetric polynomials.
    
If $g(x_1, \ldots, x_n) = g_1 + g_2 + \cdots + g_m$ is a sum of $m$ monic monomials, then
$(f \circ g)(x_1, \ldots, x_n) = f(g_1, \ldots, g_m)$.
\end{lem}

\section{Observations on Schur log-concavity}\label{sec:observations}

The notion of Schur log-concavity makes sense in a few different contexts (symmetric polynomials, symmetric functions, characters of $\SL_n$). Hence we take a more general approach.
\begin{de}
    Let $\mathcal{A}_R$ be an algebra over an ordered ring $R$.
    Given a basis $B$ of $\mathcal{A}_R$, we say
    an element $a$ is \emph{$B$-positive} (and write $a \ge_B 0$) if the coefficients in the expansion $a = \sum_{b\in B} c_b b$ are all non-negative. The notation $a_1 \ge_B a_2$ stands for $a_1 - a_2 \ge_B 0$.
    A sequence
    \(
    (a_0, a_1, a_2, \ldots)
    \)
    of elements of $\mathcal{A}_R$
    is
\emph{$B$-log-concave} if
    \[
    a_{n}^2 \stackrel[B]{}{\ge} a_{n+1}a_{n-1},
    \]
    and
\emph{strongly $B$-log-concave} if
    \[
    a_{n+i}a_n \stackrel[B]{}{\ge} a_{n+i+1}a_{n-1},
    \]
    for all $n\ge1$ and all $i\ge0$.
\end{de}

For the present paper, we mainly study the following pairs $(\mathcal{A}_\Q, B)$:
\begin{enumerate}
    \item $\big(\Lambda,~~ \{s_\lambda \ : \ \lambda\text{ a partition}\}\big),$\smallskip
    \item $\big(\Lambda[x_1, \ldots, x_n],~~ \{s_\lambda \ : \ \ell(\lambda)\le n\}\big),$\smallskip
    \item $\big(\Lambda[q,q^{-1}],~~ \{[n] \ : \ n\ge1\}\big)$.
\end{enumerate}
In each of the cases, we refer to the basis as the \emph{Schur basis}, and write $\ge_s0$ for ``Schur positive''.
\begin{lem}\label{lem:homo}
Let $\mathcal{A}_R$ and $\mathcal{A}'_R$ be algebras over $R$, with bases $B$ and $B'$,
respectively. Suppose $b_1 b_2 \ge_B 0$ for all $b_1, b_2 \in B$, and $b'_1 b'_2 \ge_{B'} 0$ for all $b'_1, b'_2 \in B'$.
Let $\pi:\mathcal{A}_R \to \mathcal{A}'_{R}$ be an algebra homomorphism such that $\pi(b) \ge_{B'} 0$ for all $b\in B$.
Then $\pi$ sends (strongly) $B$-log-concave sequences to (strongly) $B'$-log-concave sequences.
\end{lem}
\begin{proof}
     Consider a $B$-positive expansion $a_{n+i}a_n - a_{n+i+1}a_{n-1} = \sum_{b\in B} c_b b$. Taking the image under $\pi$, we get
    $\pi(a_{n+i})\pi(a_n) - \pi(a_{n+i+1})\pi(a_{n-1}) = \sum_{b\in B} c_b \pi(b)$. Since each $\pi(b)$ is $B'$-positive, so is the right-hand side.
\end{proof}

Examples of such homomorphisms include
\begin{enumerate}
    \item $\omega : \Lambda \to \Lambda, \quad s_\lambda \mapsto s_{\lambda'}$,
    \item $\Lambda \to \Lambda[\mathbf{x}], \quad s_\lambda \mapsto s_\lambda(\mathbf{x})$,
    \item $\Lambda \to \Lambda, \quad s_\lambda \mapsto s_\lambda \circ g$ for any fixed Schur positive $g\in\Lambda$,
    \item $\Lambda[x_1,x_2] \to \Lambda[q,q^{-1}], \quad s_\lambda(x_1,x_2) \mapsto s_\lambda(q,q^{-1})$.
\end{enumerate}
as well as their compositions. See \cite[Thm.~2.5.1]{BergeronNotes} for more examples.

\subsection{Properties and non-properties of Schur log-concavity}
A finite\break sequence $f_0, f_1, f_2, \ldots, f_n$ is said to be $B$-unimodal if 
\[
f_0 \le_B f_1 \le_B \cdots \le_B f_k \ge_B f_{k+1} \ge_B \cdots \ge_B f_n
\]
for some $k$.
Note that, unlike in the case of sequences of natural numbers, Schur log-concavity does not imply Schur unimodality.
\begin{ejemplo}[\sc Log-concave but not unimodal]
    The sequence $1, s_1, s_2$ is Schur log-concave. But $1 - s_1$ is not Schur positive nor Schur negative.

    Another example, in which all elements are of the same degree, is $s_{4}, s_{3,1}, s_{2,2}$.
\end{ejemplo}

For sequences of integers, the notion of ``strong log-concavity'' is equivalent to that of ``log-concavity''.
We now show by example that for our three algebras $\Lambda, \Lambda[\mathbf{x}]$ and $\Lambda[q,q^{-1}]$, the notion of strong Schur log-concavity is indeed stronger than the notion of Schur log-concavity. 
\begin{ejemplo}[\sc Weak but not strong log-concavity in $\Lambda$]
    Consider the sequence
    \[
    3s_{2,1}, ~~~
    2(s_{3}+s_{2,1}+s_{1,1,1}), ~~~
    2(s_{3}+s_{2,1}+s_{1,1,1}), ~~~
    3s_{2,1}.
    \]
Computer algebra reveals that it is
Schur log-concave, but not strongly so.
To give an example: $s_{3,3}$ appears with multiplicity $-1$ in $4(s_{3}+s_{2,1}+s_{1,1,1})^2 - 9s_{2,1}^2$.
\end{ejemplo}
\begin{ejemplo}[\sc Weak but not strong log-concavity in {$\Lambda[\mathbf{x}]$}]
\label{ej:weak and strong in polys}
    Let $\mathbf{x} = (x_1,x_2)$ and consider the sequence
    \[
    s_{3,3}(\mathbf{x}),~~~
    s_{4,2}(\mathbf{x}),~~~
    s_{5,1}(\mathbf{x}),~~~
    s_{3,3}(\mathbf{x}).
    \]
    It is Schur log-concave:
    \begin{align*}
        s_{4,2}(\mathbf{x})^2 - s_{3,3}(\mathbf{x})s_{5,1}(\mathbf{x}) &=
        s_{7,5}(\mathbf{x}) + s_{6,6}(\mathbf{x}) & \stackrel[s]{}{\ge}0,\\
        s_{5,1}(\mathbf{x})^2 - s_{4,2}(\mathbf{x})s_{3,3}(\mathbf{x}) &=
        s_{10,2}(\mathbf{x}) + s_{9,3}(\mathbf{x}) + s_{8,4}(\mathbf{x}) + s_{6,6}(\mathbf{x}) & \stackrel[s]{}{\ge}0.
    \end{align*}
    However, it is not strongly Schur log-concave:
    \begin{align*}
        s_{5,1}(\mathbf{x}) ~ s_{4,2}(\mathbf{x})
        =~& s_{9,3}(\mathbf{x}) + s_{8,4}(\mathbf{x}) + s_{7,5}(\mathbf{x}) 
        \not\stackrel[s]{}{\ge} 
        s_{6,6}(\mathbf{x}) 
        = s_{3,3}(\mathbf{x})^2.
    \end{align*}
\end{ejemplo}
\begin{ejemplo}[\sc Weak but not strong log-concavity in {$\Lambda[q,q^{-1}]$}]\label{ej:strong}
    We take the sequence from Example \ref{ej:weak and strong in polys} and specialise $\mathbf{x} \mapsto (q,q^{-1})$ to obtain the sequence
    $[1], [3], [5], [1]$. By the
Clebsch--Gordan rule,
    \begin{align*}
    [3]^2 - [5][1] &= ([5]+[3]+[1])-[5]
    \quad\text{and}\\
    [5]^2 - [3][1] &= ([9]+[7]+[5]+[3]+[1]) - [3]
    \end{align*}
    are Schur positive, and therefore the sequence is Schur log-concave. But
    \[
    [3][5] - [1]^2 = ([7]+[5]+[3]) - [1]
    \]
    which is not Schur positive: the sequence is not strongly Schur log-concave.
\end{ejemplo}

\section{Proofs of Theorem \ref{main:first row} and its corollaries}\label{sec:main}
\setcounter{mainthm}{0}
\setcounter{coro}{0}
\begin{mainthm}
The following sequence is strongly Schur log-concave:
\[s_{\lambda},~~
    s_{\lambda\cup(1^j)+(k)},~~
    s_{\lambda\cup(1^{2j})+(2k)},~~
    s_{\lambda\cup(1^{3j})+(3k)},~~ \ldots\] 
for any $k = 0$ or $k \ge \lambda_2$, and any $j = 0$ or $j\ge  \lambda'_2$.
\end{mainthm}
\begin{proof}
    Fix $k \ge \lambda_2$, $j \ge \lambda'_2$, $n\ge0$ and $i\ge0$. The cases in which $k=0$ or $j=0$ will follow immediately from the proof of the more general case.
    We aim to show
    \[
    s_{\begin{tikzpicture}[x=.6em, y=-.8em, baseline=-1.6em, fill=white] \tiny
        \filldraw (2,-1) rectangle++(6,1);
        \filldraw (0,1) rectangle++(1.2,4);
        \filldraw[x=.3em] (0,-1) -- (6,-1) -- (6,0) -- (5,0) -- (5,.5) -- (3.33,.5) -- (3.33,1) -- (3.33,1) -- (0,1) -- (0,-1);
        \node (*) at (5,-.5) {$(n+i)k$};
        \node[rotate=90] (*) at (.6,3) {$(n+i)j$};
        \node (*) at (.7,0) {$\lambda$};
    \end{tikzpicture}}
    s_{\begin{tikzpicture}[x=.6em, y=-.8em, baseline=-1.6em, fill=white] \tiny
        \filldraw (2,-1) rectangle++(5,1);
        \filldraw (0,1) rectangle++(1.2,3.5);
        \filldraw[x=.3em] (0,-1) -- (6,-1) -- (6,0) -- (5,0) -- (5,.5) -- (3.33,.5) -- (3.33,1) -- (3.33,1) -- (0,1) -- (0,-1);
        \node (*) at (4.5,-.5) {$nk$};
        \node[rotate=90] (*) at (.6,2.5) {$nj$};
        \node (*) at (.7,0) {$\lambda$};
    \end{tikzpicture}}
    \stackrel[s]{}{\ge}
    s_{\begin{tikzpicture}[x=.6em, y=-.8em, baseline=-1.6em, fill=white] \tiny
        \filldraw (2,-1) rectangle++(7,1);
        \filldraw (0,1) rectangle++(1.2,5);
        \filldraw[x=.3em] (0,-1) -- (6,-1) -- (6,0) -- (5,0) -- (5,.5) -- (3.33,.5) -- (3.33,1) -- (3.33,1) -- (0,1) -- (0,-1);
        \node (*) at (5.5,-.5) {$(n+i+1)k$};
        \node[rotate=90] (*) at (.6,3.5) {$(n+i+1)j$};
        \node (*) at (.7,0) {$\lambda$};
    \end{tikzpicture}}
    s_{\begin{tikzpicture}[x=.6em, y=-.8em, baseline=-1.6em, fill=white] \tiny
        \filldraw (2,-1) rectangle++(4,1);
        \filldraw (0,1) rectangle++(1.2,3);
        \filldraw[x=.3em] (0,-1) -- (6,-1) -- (6,0) -- (5,0) -- (5,.5) -- (3.33,.5) -- (3.33,1) -- (3.33,1) -- (0,1) -- (0,-1);
        \node (*) at (4,-.5) {$(n\!\!-\!\!1)k$};
        \node[rotate=90] (*) at (.6,2.5) {$(n\!\!-\!\!1)j$};
        \node (*) at (.7,0) {$\lambda$};
    \end{tikzpicture}}\,.
    \]
Let $\mu, \nu, \rho, \delta$ denote the four partitions in the above equation, in order. An equivalent restatement is that for all $\theta$, we have 
to show the inequality $c^\theta_{\mu\nu} \ge c^\theta_{\rho\delta}$.
of Littlewood--Richardson coefficients.
    It suffices to see that there is an injective map of Littlewood--Richardson tableaux
    \[
    \LR^\theta_{\mu\nu}\longleftarrow\LR^\theta_{\rho\delta}.
    \]
    The map is defined as follows:
    
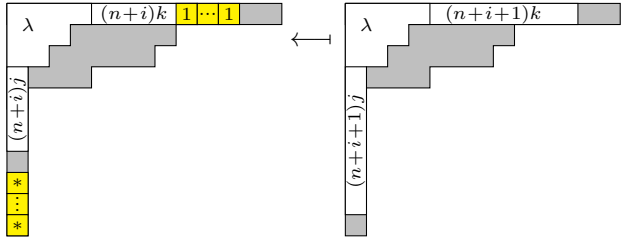
\[    
    \begin{tikzpicture}[x=.8em, y=-.8em, baseline=-1.6em] \scriptsize
        \filldraw[fill=lightgray] (0,0) -- (13,0) -- (13,1) -- (8,1) -- (8,2) -- (7,2) -- (7,3) -- (5,3) -- (4,3) -- (4,4) -- (1,4) -- (1,11) -- (0,11) -- (0,0);
        \filldraw[fill=white] (0,0) rectangle++(9,1);
        \filldraw[fill=white] (0,0) -- (4,0) -- (4,1) -- (3,1) -- (3,2) -- (2,2) -- (2,3) -- (1,3) -- (1,5) -- (0,5) -- (0,0);
        \node (*) at (6,0.5) {$(n\!+\!i)k$};
        \filldraw[fill=white] (0,3) rectangle++(1,4);
        \node[rotate=90] (*) at (.5,5) {$(n\!+\!i)j$};
        \node (*) at (1,1) {$\lambda$};
        
        \filldraw[fill=yellow] (8,0) rectangle++(1,1);
        \node (*) at (8+.5,.5) {$1$};
        \filldraw[fill=yellow] (9,0) rectangle++(1,1);
        \node (*) at (9+.5,.5) {...};
        \filldraw[fill=yellow] (10,0) rectangle++(1,1);
        \node (*) at (10+.5,.5) {$1$};
        
        \filldraw[fill=yellow] (0,8) rectangle++(1,1);
        \node (*) at (.5,8+.5) {$*$};
        \filldraw[fill=yellow] (0,9) rectangle++(1,1);
        \node[rotate=90] (*) at (.5,9+.5) {...};
        \filldraw[fill=yellow] (0,10) rectangle++(1,1);
        \node (*) at (.5,10+.5) {$*$};
    \end{tikzpicture}
    \longmapsfrom
    \begin{tikzpicture}[x=.8em, y=-.8em, baseline=-1.6em] \scriptsize
        \filldraw[fill=lightgray] (0,0) -- (13,0) -- (13,1) -- (8,1) -- (8,2) -- (7,2) -- (7,3) -- (5,3) -- (4,3) -- (4,4) -- (1,4) -- (1,11) -- (0,11) -- (0,0);
        \filldraw[fill=white] (0,0) rectangle++(11,1);
        \filldraw[fill=white] (0,0) -- (4,0) -- (4,1) -- (3,1) -- (3,2) -- (2,2) -- (2,3) -- (1,3) -- (1,5) -- (0,5) -- (0,0);
        \filldraw[fill=white] (0,3) rectangle++(1,7);
        \node[rotate=90] (*) at (.5,6.5) {$(n\!+\!i\!+\!1)j$};
        \node (*) at (7,0.5) {$(n\!+\!i\!+\!1)k$};
        \node (*) at (1,1) {$\lambda$};
    \end{tikzpicture}
    \]
    \ytableausetup{smalltableaux,centertableaux}%
    Let $T$ be a Littlewood--Richardson tableau of shape $\theta / (\lambda\cup(1^{(n+i+1)j})+((n+i+1)k))$ and content $\lambda\cup(1^{(n-1)j})+((n-1)k)$. The map has two steps: (i) insert $k$ copies of $1$ in the first row of the tableau, and (ii) insert
    \[
    \lambda'_1 + (n-1)j + 1, ~~
    \lambda'_1 + (n-1)j + 2, ~~
    \ldots, ~~
    \lambda'_1 + nj
    \]
    in the first column. (In the diagram above, the relevant cells are filled with stars.)

    The map is clearly injective; we conclude by showing that the map is well-defined. Note that if the image of the map is semistandard then it is of the correct shape and content, and that the reverse reading word is trivially Yamanouchi. So it suffices to check that it is semistandard. 

    The image of $T$ under (i) is semistandard unless $T$ has a $\ytableaushort{1}$ in row~$2$ and column $\lambda_1+ (n+i)k+1$.
    The greatest column in which a $\ytableaushort{1}$ can exist in row $2$ is
    \[
    \lambda_2+\underbrace{\lambda_1+(n-1)k}_{\#\ytableaushort{1}\text{s}}.
    \]
    But if
    \(
    \lambda_2+(\lambda_1+(n-1)k) \ge \lambda_1+(n+i)k+1
    \)
    then $\lambda_2 \ge (i+1)k + 1 > k$. This contradicts the hypothesis of the theorem. Hence the tableau $P$ obtained by applying (i) to $T$ is semistandard.

    The image of $P$ under (ii) is clearly semistandard if $T$ has no entries in column~$2$ and row $\lambda'_1+ (n+i)j+1$.
    The total number of entries in column $2$ is at most
    \[
    \lambda'_2 + \underbrace{\lambda'_1 + (n-1)j}_{\#\text{~of distinct entries}}.
    \]
    We reach a contradiction $\lambda'_2 > j$ as before.
\end{proof}
\begin{cor}\label{cor:e}
    The sequence $e_0, e_1, e_2, \ldots$ is strongly Schur log-concave.
\end{cor}
\begin{proof}
    This is the case $\lambda = (1)$, $k=0$, $j=1$ of the theorem.
\end{proof}
Another proof of the corollary is by observing that
\[
e_{n+i}e_n - e_{n+i+1}e_{n-1} = \det\left(\begin{smallmatrix}
    e_{n+i} & e_{n+i+1}\\ e_{n-1} & e_n
\end{smallmatrix}\right) =
s_{(2^n,1^i)},
\]
by the dual Jacobi--Trudi identity.
Note that this gives a second combinatorial proof of the corollary by path enumeration arguments
(cf.\ \cite{Gessel-Viennot, FK}).
The next corollary can be shown similarly using the usual Jacobi--Trudi identity.

\begin{cor}\label{cor:h}
    The sequence $h_0, h_1, h_2, \ldots$ is strongly Schur log-concave.
\end{cor}
\begin{proof}
This follows from Lemma \ref{lem:homo} after applying $\omega$ to the previous corollary.
\end{proof}

An interesting family of corollaries of Theorem \ref{main:first row} concerns the case in which $\lambda$ is a hook. For instance:
\begin{cor}\label{double-hooks}
    For all $n\ge0$ and $m\ge0$, the sequence $s_{(n,1^m)}, s_{(n+1,1^{m+1})},\break s_{(n+2,1^{m+2})}, \ldots$ is strongly Schur log-concave.
\end{cor}
\begin{proof}
    This is the case $\lambda = (n,1^m)$, $k=1$, $j=1$ of the theorem.
\end{proof}

\begin{cor}[\cite{BergeronBiagioliRosas, Lam}]
    For all $n\ge0$, the sequence $s_{(n,1^0)}, s_{(n,1^1)},\break s_{(n,1^2)}, \ldots$ is strongly Schur log-concave.
\end{cor}
\begin{proof}
    This is the case $\lambda = (n)$, $k=1$, $j=0$ of the theorem.
\end{proof}

\section{Further evidence for Conjecture \ref{main conjecture}}\label{sec:further}
In \cite{McNamara}, McNamara gives necessary conditions for partitions $\mu, \nu, \rho, \delta$ to satisfy $s_\mu s_\nu - s_\rho s_\delta \ge_s 0$. 
For fixed $\lambda, \beta$, and $\alpha$, let
\[
s_{\lambda^{(0)}}, ~~ 
s_{\lambda^{(1)}}, ~~ 
s_{\lambda^{(2)}}, ~~ 
s_{\lambda^{(3)}}, ~~ \ldots
\]
be a shorthand for the sequence of Conjecture~\ref{main conjecture}. 
As further supporting evidence for the conjecture, we show in Proposition~\ref{p:McNamara's} that the expressions 
\begin{equation}
\label{eq:expression from conjecture}
s_{\lambda^{(n)}} s_{\lambda^{(n+i)}} - s_{\lambda^{(n-1)}} s_{\lambda^{(n+i+1)}}
\end{equation}
satisfy one of these necessary conditions.
In Proposition~\ref{p:Gut-Rosas' condition}, we derive another necessary condition for Schur positivity of $s_\mu s_\nu - s_\rho s_\delta \ge_s 0$. We show in
Proposition~\ref{p:Gut-Rosas'} that it also holds for the expression \eqref{eq:expression from conjecture}
    in the case $\beta=\varnothing$.

\subsection{The condition of McNamara}
\begin{de}[\sc Dominance order]
    Given partitions $\mu$ and $\nu$, we write $\mu \preceq \nu$ and say $\nu$ \emph{dominates} $\mu$ if $\mu_1+\cdots+\mu_j \le \nu_1+\cdots+\nu_j$ for all $j$.
\end{de}
\begin{p}[{\cite[Prop.~3.1]{McNamara}}]
    If $s_\mu s_\nu \ge_s s_\rho s_\delta$ then $\mu \cup \nu \preceq \rho \cup \delta$.
\end{p}
We require a combinatorial description of the dominance order. An \emph{inner corner} of $\lambda$ is a cell of $Y(\lambda)$ such that $Y(\lambda)\setminus\{c\}$ is the set of cells of a partition. An \emph{outer corner} of $\lambda$ is a tuple $c \in \N^2$ such that $Y(\lambda)\cup\{c\}$ is the set of cells of a partition.
\begin{p}[{\cite[Prop.~2.3]{Brylawski}}]
\label{lem:dominance}
Given partitions $\mu$ and $\nu$ of the same size, $\mu \preceq \nu$ if and only if one can reach
$\mu$ by starting with $\nu$ and successively moving an inner corner to the next outer corner immediately below.
\end{p}
For instance, $(4,2,1^3)\prec(4,3,2)$ since we have the following sequence of moves:
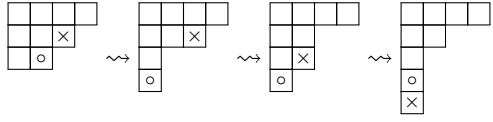
\[
\ytableaushort{{}{}{}{},{}{}{\times},{}{\circ},\none,\none}
\leadsto
\ytableaushort{{}{}{}{},{}{}{\times},{},{\circ},\none}
\leadsto
\ytableaushort{{}{}{}{},{}{},{}{\times},{\circ},\none}
\leadsto
\ytableaushort{{}{}{}{},{}{},{},{\circ},{\times}}
\]
A \emph{composition} is a tuple of non-negative integers. Given compositions $\rho$ and $ \delta$, we let $\Sort(\rho, \beta)$ be the partition obtained by sorting the parts of $\rho$ and $\delta$ in weakly decreasing order. Note that, if $\alpha$ and $\delta$ are partitions, then $\Sort(\rho,\beta) = \rho\cup\delta$.
\begin{lem}\label{lem:aux necessary}
    Let $\mu, \nu, \rho, \delta$ be compositions of length $n$ satisfying
    \[
    \mu_i+\nu_i = \rho_i+\delta_i
    \quad \text{and}\quad
    \rho_i \le \mu_i \le \nu_i \le \delta_i \quad \text{for all }i.
    \]
    Then $\Sort(\mu, \nu) \preceq \Sort(\rho, \delta)$.
\end{lem}
\begin{proof}
For $1 \le i \le n$, we have
\begin{multline*}    
\Sort\!\Big(
(\mu_1, \ldots, \mu_{i-1}, \mu_i, \rho_{i+1}, \ldots, \rho_{n}),
(\nu_1, \ldots, \nu_{i-1}, \nu_i, \delta_{i+1}, \ldots, \delta_{n})
\Big)\\
\preceq
\Sort\!\Big(
(\mu_1, \ldots, \mu_{i-1}, \rho_i, \rho_{i+1}, \ldots, \rho_{n}),
(\nu_1, \ldots, \nu_{i-1}, \delta_i, \delta_{i+1}, \ldots, \delta_{n})
\Big)
\end{multline*}
by Lemma \ref{lem:dominance}, since one can reach the former by moving
$\delta_i-\nu_i$ $(=\mu_i - \rho_i)$ corners of the latter down.
By transitivity of the dominance order, we obtain the claim.
\end{proof}
The next proposition shows that the differences of products of Schur functions involved in Conjecture~\ref{main conjecture} satisfy the necessary condition of McNamara.
Recall that $\lambda^{(n)}$ is shorthand for $\lambda \cup^n \beta + n\alpha$ when $\alpha$ and $\beta$ are clear from context.
\begin{p}\label{p:McNamara's}
Fix partitions $\lambda, \beta$ and an integral vector $\alpha$ such that $\beta_1\le\lambda_{\ell(\lambda)}$ and $\ell(\alpha)<\ell(\lambda)$.
Then
\[\lambda^{(n)} \cup \lambda^{(n+i)} \preceq \lambda^{(n-1)} \cup \lambda^{(n+i+1)}\]
for all $n$ and $i$ such that the $\lambda^{(*)}$ are partitions.
\end{p}
\begin{proof}
Consider the four compositions $\mu, \nu, \rho, \delta$ of 
length $\ell(\lambda)$ with parts
\begin{align*}
    \mu_j &= \min\!\Big\{\lambda^{(n)}_j, ~ \lambda^{(n+i)}_j\Big\}, &
    \rho_j &= \min\!\Big\{\lambda^{(n-1)}_j, ~ \lambda^{(n+i+1)}_j\Big\},
    \\
    \nu_j &= \max\!\Big\{\lambda^{(n)}_j, ~ \lambda^{(n+i)}_j\Big\}, & 
    \delta_j &= \max\!\Big\{\lambda^{(n-1)}_j, ~ \lambda^{(n+i+1)}_j\Big\}.
\end{align*}   
It is routine to check that they
satisfy the hypotheses of Lemma~\ref{lem:aux necessary}. Hence $\Sort(\mu,\nu)\preceq\Sort(\rho,\delta)$.
The claim is then that $\Sort(\mu,\nu)\cup^{2n+i}\beta\preceq\Sort(\rho,\delta)\cup^{2n+i}\beta$, which is clear since $\beta_1 \le \Sort(\mu,\nu)_{2\ell(\lambda)}$ and $\beta_1 \le \Sort(\rho,\delta)_{2\ell(\lambda)}$ by hypothesis.
\end{proof}

\subsection{The condition of Rosas and the first author}
Let $\Out(\mu)$ be the set of outer corners of $\mu$. 
Consider the sum in $\N^2$ defined by
\[
(i,j) + (\ell,k) := (i+\ell-1,j+k-1).
\]
Given a partition $\lambda$, its complement $Y(\lambda)^c$ is an ideal of $\N^2$ with respect to this sum. Given two sets $A$ and $B \subseteq \N^2$, let $A + B = \{a+b \ : \ a \in A, ~ b\in B\}$.
\begin{thm}[{\cite[Thm.~3.5]{GutierrezRosas}}]
    If $c_{\mu\nu}^\theta>0$ then $Y(\theta)^c \supseteq\Out(\mu)+\Out(\nu)+\N^2$.
\end{thm}
\begin{p}\label{p:Gut-Rosas' condition}
    Suppose $s_\mu s_\nu \ge_s s_\rho s_\delta$. Then, for each $\theta$ such that $c^\theta_{\rho,\delta} > 0$, we have $Y(\theta)^c \supseteq \Out(\mu) + \Out(\nu) + \N^2$.
\end{p}
\begin{proof}
    Since $s_\mu s_\nu \ge_s s_\rho s_\delta$ we deduce $c_{\mu\nu}^\theta \ge c_{\rho\delta}^\theta$ for all $\theta$. In particular, if $c_{\rho\delta}^\theta>0$ then $c_{\mu\nu}^\theta>0$, which by Theorem~\ref{p:Gut-Rosas' condition} implies the claim.
\end{proof}
The following result supports Conjecture~\ref{main conjecture} 
in the case $\beta = \varnothing$. After applying Lemma \ref{lem:homo}, it also applies to the case $\alpha=\varnothing$. 
\begin{p}\label{p:Gut-Rosas'}
    Fix a partition $\lambda$ and an integral vector $\alpha$.
    For each $\theta$,
    \[
    \text{if}\quad c_{\lambda+(n-1)\alpha,~\lambda+(n+i+1)\alpha}^\theta > 0
    \quad\text{then}\quad
    Y(\theta)^c \supseteq \Out(\lambda + n\alpha) + \Out(\lambda+(n+i)\alpha) + \N^2
    \]
    for all $n$ and $i$ such that all of the relevant sums are partitions.
\end{p}
\begin{proof}
Note that for any partition $\mu$ we have $\Out(\mu) + \N^2 = \{(j,\mu_j+1)\}_{j \le \ell(\mu)} + \N^2$. We define four partitions
\[
\rho = \lambda+(n-1)\alpha, ~~
\mu = \lambda+n\alpha, ~~
\nu = \lambda+(n+i)\alpha, ~~
\delta = \lambda+(n+i+1)\alpha.
\]
Assume $c^\theta_{\rho\delta}>0$. Then,
\begin{align*}
Y(\theta)^c
 &\supseteq
\{(j,\rho_j+1)\}_{j} + \{(\ell,\delta_\ell+1)\}_{\ell \le n} + \N^2\\
&= 
\{(j+\ell-1,\rho_j+\delta_\ell+1)\}_{j, \ell} + \N^2.
\end{align*}
We aim to show
\[
\{(j+\ell-1,\rho_j+\delta_\ell+1)\}_{j, \ell} + \N^2
\supseteq
\{(j+\ell-1,\mu_j+\nu_\ell+1)\}_{j, \ell}\,,
\]
which implies the claim.

The containment follows from the following observation: for all $j$ and $\ell$,
\begin{enumerate}
    \item if $\alpha_j \ge \alpha_\ell$ then
    \(
    \mu_j+\nu_\ell \ge \rho_j+\delta_\ell,
    \)
    \item if $\alpha_j < \alpha_\ell$ then
    \(
    \mu_j+\nu_\ell \ge \rho_\ell+\delta_j.
    \) \qedhere
\end{enumerate}
\end{proof}

\section{The quantum Pascal triangle}\label{sec:quantum}

In this section we show the results concerning quantum binomials.

\setcounter{mainthm}{0}
\begin{coro}
    For all $n$, the $n$th row $\Qbinom{n}{0}, \Qbinom{n}{1}, \ldots, \Qbinom{n}{n}$ of the quantum Pascal triangle is strongly Schur log-concave.
\end{coro}
\begin{proof}
  Since $\Qbinom{n}{k}$ is the image of $e_k\circ s_{(n-1)}$ in the Grothendieck group of the category of $\SL_2$ representations,
the lemma follows from Corollary~\ref{cor:e} and Lemma~\ref{lem:homo}
by the composition $f \mapsto (f\circ s_{(n-1)})(q,q^{-1})$. 
\end{proof}
Note that a direct proof using the dual Jacobi--Trudi identity is also possible, and gives
\begin{equation}\label{eq:Qbin-e}
   \left( s_{(\lambda_1,\lambda_2)'}\circ s_{(n-1)}\right) (q, q^{-1})
    = \Qbinom{n}{\lambda_1}\Qbinom{n}{\lambda_2} - \Qbinom{n}{\lambda_1+1}\Qbinom{n}{\lambda_2-1}.
\end{equation}
\begin{coro}
    For all $k$, the $k$th column $\Qbinom{k+0}{k}, \Qbinom{k+1}{k}, \Qbinom{k+2}{k}, \ldots$ of the quantum Pascal triangle is strongly Schur log-concave.
\end{coro}
\begin{proof}
Since $\Qbinom{k+\ell}{k}$ is the image of $h_\ell\circ s_{(k)}$ in the Grothendieck group of the category of $\SL_2$ representations,
the lemma follows from Corollary~\ref{cor:h} and Lemma~\ref{lem:homo} by the
composition $f \mapsto (f\circ s_{(k)})(q,q^{-1})$.
\end{proof}
A direct proof using the Jacobi--Trudi identity gives
\begin{equation}\label{eq:Qbin-h}
   \left( s_{(\lambda_1,\lambda_2)} \circ s_{(k)}\right) (q,q^{-1}) = \Qbinom{k+\lambda_1}{k}\Qbinom{k+\lambda_2}{k} - \Qbinom{k+\lambda_1+1}{k}\Qbinom{k+\lambda_2-1}{k}.
\end{equation}
\begin{note}[\sc $q$-analogues]
    \label{note:q-analogue}
    The $q$-analogue of Corollary \ref{thm:cols} follows similarly:
    \begin{quote}
    \itshape
    For all $k$, the $k$th column
    $\qbinom{k+0}{k}, \qbinom{k+1}{k}, \qbinom{k+2}{k}, \ldots$
    of the $q$-Pascal triangle is strongly $q$-log-concave. 
    \end{quote}
    Indeed, the $\ell$th term is $(h_\ell\circ s_{(k)})(1,q)$, and $f\mapsto f(1,q)$ sends Schur functions to polynomials with positive coefficients. Hence
the claim follows from the Schur log-concavity of $h_0, h_1, h_2, \ldots$ after applying $f \mapsto (f\circ s_{(k)})(1,q)$.

    However, note that the $q$-analogue of Corollary~\ref{thm:rows}
does not follow in the same manner:
application of $f \mapsto (f\circ s_{(n-1)})(1,q)$ to the sequence
    $e_0, e_1, e_2, \ldots$ gives 
    \[
    q^{\binom{0}{2}}\qbinom{n}{0}, ~~~
    q^{\binom{1}{2}}\qbinom{n}{1}, ~~~
    q^{\binom{2}{2}}\qbinom{n}{2}, ~~~
    q^{\binom{3}{2}}\qbinom{n}{3}, ~~~\ldots
    \]
    which is not the $n$th row of the $q$-Pascal triangle. A combinatorial proof with Jacobi--Trudi determinants is however possible; see \cite{SaganLogConcave} for details.
\end{note}
\medskip

Although we have mainly been interested in sequences of irreducible characters, there is no reason to not explore other sequences.
\begin{cor}
    For any $n, m, d \ge 0$, the sequence of $\ZZ[q, q^{-1}]$ whose $k$th term is $\Qbinom{n+m+2k-1}{m+k}\Qbinom{n+d+k}{n+m+2k}$
    is strongly Schur log-concave.
\end{cor}
\begin{proof}
    A quantum version of Stanley's hook-content formula \cite[Cor.~7.21.4]{StanleyEC2} for hook shapes says
    \[
    \big(s_{(n,1^m)}\circ s_{(d)}\big)(q,q^{-1}) = \Qbinom{n+m-1}{m} \Qbinom{n+d}{n+m}.
    \]
    See for instance \cite{GMSW-hooks}. The claim then follows from Corollary \ref{double-hooks} and Lemma \ref{lem:homo} after applying $f \mapsto (f\circ s_{(d)})(q,q^{-1})$.
\end{proof}

\section{Concluding remarks}
An unpublished conjecture of Lam, Postnikov and Pylyavskyy is the following. We point to \cite{DobrovolskaPylyavskyy} for the definition of an alcoved polytope. Denote by $\chi_\lambda$ the $\SL_n$ character indexed by $\lambda$, which is the image of $s_\lambda$ in $\ZZ[x_1, \ldots, x_n]/(x_1\cdots x_n-1)$.
\begin{conjecture}
    Let $\rho, \mu, \nu, \rho$ be partitions. Suppose $\mu+\nu = \rho+\delta$, and that $\mu$ and $\nu$ belong to the minimal alcoved polytope containing $\rho$ and $\delta$. Then $\chi_{\mu}\chi_{\nu} \ge_s \chi_{\rho}\chi_{\delta}$.
\end{conjecture}
The corresponding conjecture for symmetric functions implies the following:
\begin{conjecture}
    For all partitions $\lambda$ and all integral vectors $\alpha$, the sequence
    \[
    s_{\lambda}, ~~
    s_{\lambda+\alpha}, ~~
    s_{\lambda+2\alpha}, ~~
    s_{\lambda+3\alpha}, ~~\ldots
    \]
    is strongly Schur log-concave. 
\end{conjecture}
This is the case $\beta=\varnothing$ and $\alpha\in\ZZ^n$ (without restriction on the length of $\alpha$) of Conjecture \ref{main conjecture}. We ask, more generally:
\begin{prob}
    Under which conditions on $\lambda, \beta, \alpha$ is the sequence
    \[
    s_{\lambda},~~
    s_{\lambda\cup\beta+\alpha},~~
    s_{\lambda\cup^2\beta+2\alpha}, ~~
    s_{\lambda\cup^3\beta+3\alpha}, ~~\ldots
    \]
    strongly Schur log-concave?
\end{prob}

To illustrate the subtleties in the problem, we provide a counterexample to a relaxed version of Conjecture \ref{main conjecture}.
\begin{ejemplo}
    Let $\lambda = (3,3)$, $\beta = (3)$, $\alpha = (1,1)$. 
    Note that $\ell(\alpha) = \ell(\lambda)$. The sequence
    \[
    s_{(3^2)}, ~~
    s_{(4^2,3)}, ~~
    s_{(5^2,3^2)}, ~~
    s_{(6^2,3^3)}, ~~\ldots
    \]
    is not Schur log-concave, since $s_{(4^2,3)}^2 \not\ge_s s_{(3^2)}s_{(5^2,3^2)}$.
\end{ejemplo}

\section*{Acknowledgements}
The first author would like to thank Mark Wildon for his guidance. We thank Micha{\l} Szwej and \'Alvaro Mart\'inez for helping shape the statement of the problem;
Bruce Sagan for interesting comments.

\'AG was funded by a University of Bristol Research Training Support Grant.
CK~was partially supported by the Austrian
Science Foundation FWF, grant\break 10.55776/F1002, in the framework
of the Special Research Program ``Discrete Random Structures:
Enumeration and Scaling Limits''.

\bibliographystyle{alpha}
\bibliography{qPascalDiagonal/Bib2}

\end{document}